\documentclass[letterpaper, 10 pt, conference]{ieeeconf}

\IEEEoverridecommandlockouts                              % This command is only needed if 

\usepackage{tikz}
\usepackage{textcomp}

\usetikzlibrary{matrix}
\usetikzlibrary{angles,quotes}

\usepackage{graphics,graphicx} % for pdf, bitmapped graphics files
\usepackage{epsfig} % for postscript graphics files
\usepackage{mathptmx} % assumes new font selection scheme installed
\usepackage{times} % assumes new font selection scheme installed
\usepackage{amsmath} % assumes amsmath package installed
\usepackage{amssymb,nth, mathtools,dsfont}  % assumes amsmath package installed
\usepackage{amsbsy}
\usepackage{epstopdf}
\usepackage[noadjust]{cite}
\usepackage{romannum}
\usepackage{xcolor}
\usepackage{subcaption}
\usepackage{multirow}

\newcommand{\field}[1]{\mathbb{#1}}
\newcommand{\R}{\field{R}}
\newcommand{\RH}{\field{RH}_\infty}
\newcommand{\C}{\field{C}}

\newcommand{\N}{\field{N}}
\newcommand{\Z}{\field{Z}}

\newtheorem{theorem}{Theorem}
\newtheorem{lemma}{Lemma}
\newtheorem{definition}{Definition}

\newtheorem{conjecture}{Conjecture}
\newtheorem{proposition}{Proposition}

\makeatletter
\newcommand*{\rom}[1]{\expandafter\@slowromancap\romannumeral #1@}

\title{Construction of Periodic Counterexamples to \\ the Discrete-Time Kalman Conjecture}

\author{Peter Seiler and  Joaquin Carrasco% <-this % stops a space
	\thanks{P. Seiler is with the Department of Electrical Engineering \& Computer Science, University of Michigan, Ann Arbor, US.
		{\tt\small pseiler@umich.edu}}
	\thanks{J. Carrasco is with the Department of Electrical $\&$ Electronic Engineering, University of Manchester, M13 9PL, UK. 
		{\tt\small joaquin.carrasco@manchester.ac.uk}  
	}
\thanks{This work was partially supported by EPSRC project EP/S03286X/1.}
}

\begin{document}

\maketitle

\begin{abstract}
  This paper considers the Lurye system of a discrete-time, linear
  time-invariant plant in negative feedback with a nonlinearity.  Both
  monotone and slope-restricted nonlinearities are considered.  The
  main result is a procedure to construct destabilizing nonlinearities
  for the Lurye system. If the plant satisfies a certain phase
  condition then a monotone nonlinearity can be constructed so that
  the Lurye system has a non-trivial periodic cycle.  Several examples
  are provided to demonstrate the construction. This represents a
  contribution for absolute stability analysis since the constructed
  nonlinearity provides a less conservative upper bound than existing
  bounds in the literature.
\end{abstract}

%\begin{IEEEkeywords}
% Absolute stability; Lurye system; Kalman Conjecture; Cycle behaviour. 
%\end{IEEEkeywords}

\section{Introduction}

The discrete-time absolute stability problem considers the Lurye
system of a discrete-time, linear time-invariant (LTI) plant in
negative feedback with a nonlinearity. Let $k_{AS}$ denote the
supremum of the set of values of $k$ for which the Lurye system is
stable for all nonlinearities whose slope is restricted to $[0,k]$. It
remains an open question to provide necessary and sufficient
conditions to compute this maximal stability interval $k_{AS}$.  The
LTI Zames--Falb
multipliers~\cite{OShea67,OShea:1967,Zames:1968,Willems:1968,Willems:71,Carrasco:2016}
provide a sufficient condition for stability. Specifically, the search
over discrete-time Zames-Falb multipliers in~\cite{Carrasco:20}
provides a lower bound $k_{ZF} \le k_{AS}$.  It has been conjectured
in~\cite{Shuai:2018,Zhang:20} that this condition is actually
necessary and sufficient, i.e. $k_{ZF}=k_{AS}$.  In other words, the
conjecture is that if a Zames-Falb multiplier does not exist for some
$k$ then there exists a destabilizing nonlinearity whose slope remains
within $[0,k]$.

The main contribution of this paper is a method to systematically
construct destabilizing nonlinearities for the Lurye system. Such
nonlinearities provide upper bounds $\bar{k} \ge k_{AS}$ and hence are
complementary to the Zames-Falb conditions. The construction is based
on a frequency-domain condition developed in~\cite{Zhang:20} from the
dual problem of the Zames-Falb condition. The construction is first described
for Lurye systems with monotone nonlinearities
(Section~\ref{sec:mainS0inf}). If the plant satisfies a phase
condition at one frequency then there is a monotone nonlinearity such
that the Lurye system has a non-trivial periodic solution.  The
destabilizing nonlinearity is explicitly constructed from the periodic
solution.  Next, the results are extended to Lurye systems with
slope-restricted nonlinearities via a loop transformation
(Section~\ref{sec:mainS0k}).

The only existing method to systematically construct a destabilizing
nonlinearity is, to our knowledge, given by the Nyquist criterion.
This provides the smallest linear gain, referred to as the Nyquist
gain $k_N$, that destabilizes the Lurye system
(Section~\ref{Sec:Kal}). The Nyquist gain provides another upper bound
$k_N \ge k_{AS}$ but it is known that this upper bound is
conservative. Specifically, the discrete-time Kalman conjecture is
that $k_N=k_{AS}$.  This conjecture was shown to be false in
\cite{Carrasco:15,Heath:15} and hence $k_N>k_{AS}$ in general.  Our
paper constructs destabilizing nonlinearities with slope restricted to
$[0,\bar{k}]$.  If $\bar{k}<k_N$ then the destabilizing nonlinearity
represents a counterexample to the Kalman conjecture.

It is worth noting that the construction of counterexamples of the
continuous-time Kalman Conjecture has been investigated since the
sixties.  It still attracts interest due to the ill-posed numerical
issues~\cite{Fitts:66,Barabonov:88,bragin:2011,Leonov:2013}.  For the
Aizerman conjecture, a systematic analysis of the existence of periodic cycles for second-order systems has been explored
in~\cite{Zvyagintseva:20a,Zvyagintseva:20b}. In the context of optimization, construction of nonlinearities for worst-case convergence rate has been used in~\cite{Lee:20}.

% Frequency domain conditions have been developed in \cite{Shuai:2018,Zhang:20} for the non-existence of a Zames-Falb multiplier.  However, there is no link between these frequency domain conditions and the existence of a destabilizing nonlinearity.

%This paper uses the frequency domain conditions in~\cite{Zhang:20} to construct destabilizing nonlinearities.  First, a Lurye system is considered with monotone nonlinearities.  If the plant satisfies a phase condition at one frequency then there is a monotone nonlinearity such that the Lurye system has a non-trivial periodic solution.  The destabilizing nonlinearity is explicitly constructed from the periodic solution.  Next, the results are extended to Lurye systems with slope-restricted nonlinearities.  A destabilizing nonlinearity is constructed with slope restricted to $[0,\bar{k}]$.  This construction is complementary to the Zames-Falb condition.  Specifically, the Zames-Falb condition provides a $k_{ZF}$ such that the Lurye system is stable for all nonlinearities with slope in the interval $[0,k_{ZF}]$. The destabilizing nonlinearity demonstrates that $k_{ZF} \le \bar{k}$.

\section{Notation}

The set of integers and positive, natural numbers are denoted as $\Z$
and $\N^+$, respectively.  $\RH$ denotes the space of real, rational
functions with all poles inside the open unit disk.  This space
corresponds to transfer functions for stable, LTI discrete-time
systems.  A function $\phi:\R\rightarrow \R$ has slope restricted to
$[0,k]$ for some finite $k> 0$ if:
\begin{align}
  0 \le \frac{ \phi(y_2) - \phi(y_1)}{y_2 - y_1} \le k
 \quad \forall y_2 \ne y_1
\end{align}
$S_{0,k}$ with $k<\infty$ denotes the set of all functions with slope
restricted to $[0,k]$ The notation $S_{0,k}$ with $k=\infty$
corresponds to the special case where $\phi$ is multivalued and
monotone: $y_2 \ge y_1$ implies $\phi(y_2) \ge \phi(y_1)$.  In this
case, $u \in \phi(y)$ will denote that $u$ is one of the values taken
by $\phi$ at $y$.  In addition, $S_{0,k}^{\text{odd}}$ denotes the set
of all odd functions with slope restricted to $[0,k]$,
i.e. $\phi(x)=-\phi(-x)$ for all $x\in\R$.

Finally, let $\{h_0,h_1,\ldots,h_{T-1}\}$ denote a finite sequence of
real numbers. We will often stack such sequences into a column vector
$H_T:=[h_0, \, h_1, \ldots, \, h_{T-1}]^\top \in \R^T$. The circulant
matrix for a given finite sequence $H_T$ is defined as:
\begin{equation}\label{eq:1}
C(H_T):=\begin{bmatrix}
h_0 & h_{T-1} & h_{T-2} & \cdots & h_2 & h_1\\
h_1 & h_0 & h_{T-1} & \cdots & h_3 & h_2\\
h_2 & h_1 & h_0 & \cdots & h_4 & h_3\\
\vdots & \vdots & \vdots & \ddots & \vdots & \vdots\\
h_{T-2} & h_{T-3} & h_{T-4} & \cdots & h_0 & h_{T-1}\\
h_{T-1} & h_{T-2} & h_{T-3} & \cdots & h_1 & h_0\\
\end{bmatrix}.
\end{equation} 

\section{Problem statement}

Let $G$ be a discrete-time system that is LTI and single-input,
single-output (SISO).  We consider the Lurye system of $G$ in negative
feedback with a nonlinearity $\phi:\R\rightarrow \R$ as shown in
Figure~\ref{fig:lure}. The Lurye system is expressed as
\begin{equation}\label{eq:Lure}
y=Gu, \quad u_k=-\phi(y_k),
\end{equation}
We consider functions in $S_{0,k}$ with $k>0$ and $\phi(0)=0$.  Lurye
systems with both $k<\infty$ and $k=\infty$ will be considered. These
cases are related by a loop transformation as discussed later in the
paper.  Additional details on this formulation can be found in
\cite{D&V:75}.

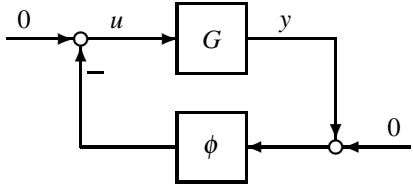
\begin{figure}[ht]
  \centering
  \ifx\JPicScale\undefined\def\JPicScale{0.3}\fi
\unitlength \JPicScale mm
\begin{picture}(180,75)(0,-65)
\thicklines
\put(0,0){\vector(1,0){30}}
\put(33,0){\circle{6}}
\put(5,5){0}

\put(36,0){\vector(1,0){40}}
\put(46,5){$u$}
\put(76,-15){\framebox(30,30){$G$}}

\put(106,0){\line(1,0){40}}
\put(121,5){$y$}
\put(146,0){\vector(0,-1){45}}
\put(146,-48){\circle{6}}

\put(179,-48){\vector(-1,0){30}}
\put(169,-43){0}
\put(143,-48){\vector(-1,0){37}}
\put(76,-63){\framebox(30,30){$\phi$}}

\put(76,-48){\line(-1,0){43}}
\put(33,-48){\vector(0,1){45}}
\put(36,-15){\line(1,0){7}}

% %\linethickness{0.3mm}
% \put(10,2.5){\line(1,0){10}}
% \put(10,2.5){\line(0,1){7.5}}
% \put(10,10){\vector(0,1){0.12}}
% \put(10,12.5){\circle{5}}

% %\linethickness{0.5mm}
% %\put(12,15){\line(1,0){2}}
% \linethickness{0.3mm}
% \put(0,12.5){\line(1,0){7.5}}
% \put(7.5,12.5){\vector(1,0){0.12}}
% \put(12.5,12.5){\line(1,0){7.5}}
% \put(20,12.5){\vector(1,0){0.12}}
% \put(30,12.5){\line(1,0){10}}
% \put(40,5){\line(0,1){7.5}}
% \put(40,5){\vector(0,-1){0.12}}

% \put(40,2.5){\circle{5}}
% \put(42.5,2.5){\line(1,0){7.5}}
% \put(42.5,2.5){\vector(-1,0){0.12}}
% \put(30,2.5){\line(1,0){7.5}}
% \put(30,2.5){\vector(-1,0){0.12}}

% \put(20,10){\line(1,0){10}}
% \put(20,15){\line(1,0){10}}
% \put(20,10){\line(0,1){5}}
% \put(30,10){\line(0,1){5}}

% \put(20,0){\line(1,0){10}}
% \put(20,5){\line(1,0){10}}
% \put(20,0){\line(0,1){5}}
% \put(30,0){\line(0,1){5}}

% \put(25,12.5){\makebox(0,0)[cc]{$G$}}
% \put(25,2.5){\makebox(0,0)[cc]{$\phi$}}

% \put(3,14){\makebox(0,0)[cc]{$0$}}
% \put(15,14){\makebox(0,0)[cc]{$u$}}
% \put(35,14){\makebox(0,0)[cc]{$y$}}
% \put(12.5,9){\makebox(0,0)[cc]{$-$}}
% \put(47,4){\makebox(0,0)[cc]{$0$}}
% %\put(35,4){\makebox(0,0)[cc]{$e_2$}}
% %\put(15,4){\makebox(0,0)[cc]{$h$}}
% %\put(35,14){\makebox(0,0)[cc]{$y$}}
\end{picture}
  \caption{Autonomous Lurye system}
  \label{fig:lure}
  \vskip-3mm
\end{figure}

We provide conditions on $G$ for the existence of non-trivial periodic
solutions to the Lurye system in Figure~\ref{fig:lure}.  Specifically,
let the plant $G$, slope constant $k>0$, and time horizon $T\in \N^+$
be given.  We provide sufficient conditions for the existence of a
nonlinearity $\phi \in S_{0,k}$ with $\phi(0)=0$ such that the Lurye
system has a non-trivial $T$-periodic solution.  If the conditions are
feasible then the proof provides a construction for the periodic
signals $U_T\in \R^T$ and $Y_T \in \R^T$.  A nonlinearity
$\phi \in S_{0,k}$ can then be constructed to interpolate $(Y_T,-U_T)$
and $(0,0)$.

\section{Preliminary Results}

This section presents two preliminary results that are used
in the derivation of the main results.

% XXX - I guess delta = 2pi also satisfies the equation 4 and
% hence delta <= pi/T is not necessary (due to periodicity). We
% need to state that we are restricting the range of delta.
% Also, I think eq:Lem1 always holds for T=1 so we need to assume T>=2.
\vspace{0.05in}
\begin{lemma}\label{lem:1}
  Let $T\in\N^+$ be given.  Then $\delta \in [-\pi,\pi]$ satisfies
  $|\delta| \leq \pi/T$ if and only if
  \begin{align}
    \label{eq:Lem1a}
    \hbox{Re}\{e^{j\delta}e^{j(\frac{ \pi}{T}k+\frac{\pi}{2})}\}
    \hbox{Re}\{e^{j(\frac{\pi}{T}k+\frac{\pi}{2})}\}\geq 0,
    \quad \mbox{for } k\in \Z.
  \end{align}
\end{lemma}
\begin{proof}
  The result is trivially true for the case $T=1$, hence the rest of
  the proof considers the case $T\ge2$. To simplify notation, define
  $z_k:=e^{j(\frac{\pi}{T}k+\frac{\pi}{2})} \in \C$.  The sequence
  $z_k$ has period $2T$ with $z_0=0$ and $z_k=-z_{k+T}$. Thus
  Equation~\ref{eq:Lem1a} is equivalent to:
  \begin{align}
    \label{eq:Lem1b}
    \hbox{Re}\{e^{j\delta} z_k \} \hbox{Re}\{z_k \}\geq 0
    \quad \mbox{for } k=1,2,\ldots, T-1
  \end{align}
  The phase of $\{z_k\}_{k=1}^{T-1}$ ranges from $\pi/2+\pi/T$ up to
  $3\pi/2-\pi/T$. Hence all values of $\{z_k\}_{k=1}^{T-1}$ have
  strictly negative real part.  It follows that
  Equation~\ref{eq:Lem1b} is equivalent to:
  $\hbox{Re}\{e^{j\delta}z_k\}\leq 0$ for $k=1,2,\ldots,T-1$.
  This can be written as the following inequality on the phase:
  \begin{align}
    \frac{\pi}{2} \le \frac{\pi}{T}k+\frac{\pi}{2}+\delta \le \frac{3\pi}{2}
    \quad \mbox{for } k=1,2,\ldots, T-1
  \end{align}
  Thus Equation~\ref{eq:Lem1a} holds if and only
  if (restricting $\delta \in [-\pi,\pi]$):
  \begin{align}
    -\frac{\pi}{T}k \le \delta \le \pi -\frac{\pi}{T}k 
    \quad \mbox{for } k=1,2,\ldots, T-1
  \end{align}
  This condition is equivalent to $|\delta|\le \pi/T$.
\end{proof}	
\vspace{0.05in}

The next result provides a necessary and sufficient condition to 
interpolate finite sequences by a multi-valued function in
$S_{0,\infty}$. This result appears in Section 8 of 
\cite{lambert04} and more general finite interpolation results
appear in \cite{taylor16} and \cite{taylor17phd}.

\vspace{0.05in}
\begin{lemma}[\cite{lambert04}]
  \label{lem:2}
  Let finite sequences $\{y_i\}_{i=0}^{T-1}$ and $\{u_i\}_{i=0}^{T-1}$
  be given.  There exists $\phi\in S_{0,\infty}$ such that
  $-u_i \in \phi(y_i)$ for $i=0,\ldots,T-1$ if and only if:
  \begin{align}
    \label{eq:monotonecond}
    (y_i - y_l)  (u_i - u_l) \le 0
    \quad \forall i,l \in \{0,\ldots, T-1\}
  \end{align}
\end{lemma}
\vspace{0.05in}

% (This lemma is written generally and assumes we'll add the (0,0) 
% point to the sequences later in the paper).
A formal proof is given in \cite{lambert04}.  If the finite sequences
satisfy Equation~\ref{eq:monotonecond} then there is, in general, more
than one $\phi \in S_{0,\infty}$ that interpolates the data. Here we
will provide an explicit formula for a $\phi \in S_{0,\infty}$ that
interpolates the data. First, re-order the points so that
$y_0 \le y_1 \le \cdots \le y_{T-1}$ and
$-u_0 \le -u_1 \le \cdots \le -u_{T-1}$. This re-ordering is possible
since the data satisfy Equation~\ref{eq:monotonecond}.  Next note that
there can be repeats in the input data: $y_i=y_{i+1}=\cdots=y_{i+r}$
for some $r>0$. In this case the nonlinearity $\phi$ is multi-valued:
$\phi(y_i) \in [-u_i,-u_{i+r}]$.  Finally, the re-ordered sequences
are interpolated by the following multi-valued function:
\begin{align}
\label{eq:phiinterp}
\phi(y) \subseteq \left\{
  \begin{array}{ll}
    -u_0 & \mbox{if } y < y_0 \\
    \left[-u_i, -u_{i+r}\right]  
         & \mbox{if } y = y_i =\cdots=y_{i+r} \\
         & \,\, \mbox{ for some } r \ge 0 \\
           (f_i-1) u_i - f_i u_{i+1} 
         &  \mbox{if } y_i < y < y_{i+1} \\
         & \,\, \mbox{ where } f_i:= \frac{y-y_i}{y_{i+1}-y_i}  \\    
    -u_{T-1} & \mbox{if } y > y_{T-1}
  \end{array} 
   \right.
\end{align}
This corresponds to linear interpolation or multi-valued output for
any input $y\in[y_0,y_{T-1}]$ and nearest neighbor extrapolation
otherwise.  This specific nonlinearity has  the following useful property:
\vspace{0.05in}
\begin{lemma}\label{lem:odd}
  Suppose the finite sequences $\{y_i\}_{i=0}^{T-1}$ and
  $\{u_i\}_{i=0}^{T-1}$ are odd, i.e.  $(y_i,u_i)$ is in the sequence
  if and only if $(-y_i,-u_i)$ is in the sequence. Then the
  nonlinearity $\phi$ in Equation~\ref{eq:phiinterp} is odd and has
  $0\in \phi(0)$.
\end{lemma}

\begin{proof}
  The proof is straightforward by construction of $\phi$
  in Equation~\ref{eq:phiinterp}.
\end{proof}

\section{Main Results}

\subsection{Construction for $S_{0,\infty}$}
\label{sec:mainS0inf}

Theorem~\ref{th:1} below provides conditions for the existence of
$\phi \in S_{0,\infty}$ such that the Lurye system has a
non-trivial $T$-periodic solution.  The proof relies on the response
of the LTI system $G\in \RH$ due to periodic inputs. Let
$g:=\{g_0,g_1,g_2,\dots\}$ denote the impulse response of $G$.
The convolution summation for a (not necessarily periodic)
input sequence $\{u_i\}_{i=-\infty}^\infty$ is:
\begin{align}
  y_k = \sum_{i=-\infty}^k  g_{k-i} u_i
\end{align}
Next, consider the case where the input is $T$-periodic so that
$u_{i+T} = u_i$ for all $i$.  The terms in convolution summation can
be re-grouped.  This yields the following $T$-periodic output
\begin{align}
   y_k =\sum_{i=0}^T h_{k-i} u_i 
   \,\, \mbox{ where } \,\, h_i := \sum_{l=0}^\infty g_{i+lT}.
\end{align}
To simplify the notation, define the column vector
$H_T:=\begin{bmatrix}h_0 & h_1 & \ldots & h_{T-1}\end{bmatrix}^\top
\in \R^T$.  Similarly, stack the $T$-periodic sequences
$\{u_i\}_{i=0}^{T-1}$ and $\{y_i\}_{i=0}^{T-1}$ into vectors $U_T$ and
$Y_T$, respectively.  The $T$-periodic inputs and outputs are related
by $Y_T = C(H_T) U_T$ where $C(H_T)$ is the circulant matrix in
Equation~\ref{eq:1}. We are now ready to state the main results.

\vspace{0.05in}
\begin{theorem}\label{th:1}
  Let $G\in\RH$ and integers $0<\alpha<\beta$ be given.  Assume
  $\alpha$ and $\beta$ are co-prime, i.e. their greatest common
  divisor is 1. Define the frequency
  $\omega:=\frac{\alpha \pi}{\beta}$ with corresponding period 
  $T=2\beta$ if $\alpha$ is odd and $T=\beta$ if $\alpha$ is even.
  There exists $\phi \in S_{0,\infty}$ such that the Lurye system
  has a non-trivial $T$-periodic solution if 
  \begin{equation}\label{eq:5}
    \pi-\frac{\pi}{T} \le
    \angle G(e^{j\omega})\le \pi+\frac{\pi}{T}.
  \end{equation} 
\end{theorem}
\begin{proof}
  Define the $T$-periodic input $U_T :=\hbox{Re}\{V_T\}$ where:
  \begin{align}
    V_T:=\begin{bmatrix}
    1 & e^{j\omega}& \ldots & e^{j\omega (T-1)}
    \end{bmatrix} \in \C^T.
  \end{align}
  Note that $V_T$ is an eigenvector of $C(H_T)$ with eigenvalue
  $G(e^{j\omega})$ \cite{Gelb:1968,Golub:13}. Hence
  $C(H_T) V_T = G(e^{j\omega}) V_T$ and the $T$-periodic output is
 $Y_T = \hbox{Re} \{ C(H_T) V_T \}
        = \hbox{Re} \{ G(e^{j\omega}) V_T \}$.

  Next, we show that the input/output sequences can be interpolated by
  a nonlinearity $\phi \in S_{0,\infty}$.  If Equation~\ref{eq:5} holds
  then $G(e^{j\omega}) = -r e^{j\delta}$ for some $r>0$ and
  $|\delta|\le \pi/T$.  Use the expressions for $U_T$, $Y_T$, and
  $G(e^{j\omega})$ to show the following:
  \begin{align*}
    (y_i-y_l)(u_i-u_l)=
    \hbox{Re}\{-r e^{j\delta}(e^{j\omega i}-e^{j\omega l})\} \, 
    \hbox{Re}\{e^{j\omega i}-e^{j\omega l}\}.
  \end{align*}
  The following identity holds for any integers $i$ and $l$:
  \begin{align*}
    e^{j\omega i}-e^{j\omega l} = 
    2\sin\left(\frac{\omega}{2}(i-l)\right)e^{j(\frac{\omega}{2}(i+l)+\frac{\pi}{2})}.
  \end{align*}
  This identity yields:
  \begin{align*}
    & (y_i-y_l)(u_i-u_l)= 
     -c\ 
    \hbox{Re}\{e^{j\delta}e^{j(\frac{\omega}{2}(i+l)+\frac{\pi}{2})}\}
    \hbox{Re}\{e^{j(\frac{\omega}{2}(i+l)+\frac{\pi}{2})} \}
  \end{align*}
  where $c:=4r\sin^2\left(\frac{\omega}{2}(i-l)\right)\geq0$.
  Finally, $\frac{\omega}{2}=\frac{\alpha\pi}{T}$ if $\alpha$ is odd
  or $\frac{\omega}{2}=\frac{\alpha\pi}{2T}$ if $\alpha$ is even.  In
  either case, $\frac{\omega}{2}(i+l)=\frac{\pi}{T} k$ for some
  integer $k$. It follows from Lemma~\ref{lem:1} that
  $(y_i-y_l)(u_i-u_l)\leq 0$ for any $i,l \in \{0,\ldots,T-1\}$.  By
  Lemma~\ref{lem:2}, there exists $\phi\in S_{0,\infty}$ such that
  $-u_i \in \phi(y_i)$ for $i=0,\ldots,T-1$.

  The only remaining issue is to show that the multi-valued function 
  satisfies $0\in \phi(0)$. There are two cases:
  
  \emph{A) $\alpha$ is odd:} The frequency is
  $\omega=\frac{2\pi\alpha}{T}$ where $T=2\beta$ is even.  The points
  in $V_T \in \C^T$: (i) are equidistantly spaced around the unit
  circle, (ii) are symmetric about both the real and imaginary axis, (iii) and there is a rotational symmetry of $\pi$. The points in
  $C(H_T) V_T = G(e^{j\omega}) V_T$ are scaled and rotated by the
  magnitude and phase of $G(e^{j\omega})$. If $G$ satisfies the phase
  constraint in \eqref{eq:5} then these points are: (i) equidistantly
  spaced around a circle, (ii) they are rotated an angle $\delta$ with respect to $V_T$, (iii) and there is a rotational symmetry of $\pi$. As a result the interpolating data is odd: if
  $(y_i,-u_i)$ is a point in the input/output data then $(-y_i,u_i)$
  is as well.  By Lemma~\ref{lem:odd}, the interpolating nonlinearity
  is not only monotone but is also odd and satisfies
  $0\in \phi(0)$.

  % XXX - Need to check with JC on the statements above as they
  %  don't seem to perfectly match properties 1-6 in his proof
  %  of Corollary 1, e.g. I don't think yi=0 is possible which
  %  he has as property 1. But then why have property 3?
  % XXX - Is the staircase function given earlier actually odd?
  %  Or does the definition need to be modified?
  % XXX - What is happening if the phase of G is strictly
  %  in inside [pi-pi/T,pi+pi/T]?  Most of our discussions focused
  %  on the case where the phase is exactly pi \pm pi/T. 
  %  The corollaries below seem to be written with this 
  %  phase equality constraint but this main theorem allows
  %  inequality.  It seems like if the phase constraint is
  %  is strictly satisfied then we just get a larger class
  %  of possible counterexamples and the points (+/- yi,0)
  %  form a smaller interval around the origin. Does this have
  %  any particular importance?

  \emph{B) $\alpha$ is even:} The frequency is
  $\omega=\frac{\alpha\pi}{T}$ where $T=\beta$ is odd.  The points in
  $V_T \in \C^T$ are again equidistantly spaced around the unit circle
  and symmetric about the real axis.  However, the rotational symmetry
  of $\pi$ no longer holds and hence the sequence of points is not
  odd. As a result, the interpolated function is not odd.  This is an
  expected property from the analysis in~\cite{Zhang:20} for the case
  where $\alpha$ is even. More importantly, the interpolated function
  fails to satisfy $0\in \phi(0)$.  It is possible to shift the
  nonlinearity to recover $0\in \phi(0)$. First, modify the definition
  of the input sequence to be
  $\hat{U}_T=\hbox{Re}\{V_T\}+\xi\mathbf{1}$ where
  $\mathbf{1} \in \R^T$ is a vector of ones and $\xi$ is to be
  chosen. Note that $C(H_T)\mathbf{1} = G(1) \mathbf{1}$ where
  $G(1) = \sum_{k=0}^\infty g_k$ is the DC gain of the system. Thus
  the modified output sequence is:
  \begin{align}\label{eq:2}
    \hat{Y}_T = \hbox{Re}\{ C(H_T) \hat{U}_T \}
        = \hbox{Re}\{ G(e^{j\omega}) V_T \} + \xi G(1) \mathbf{1}
  \end{align}
  This modification adds the constants $\xi$ and $\xi G(1)$ to the 
  input and output sequences, respectively.  The choice of
  $\xi$ shifts the original curve generated by $(Y_T,-U_T)$ along the
  line connecting $(0,0)$ and $(G(1),-1)$.  Find the intersection
  of the original curve with the line connecting $(0,0)$ and $(G(1),-1)$.  
  This yields the value of $\xi$ so that the modified function
  satisfies $0\in \phi(0)$.  This function is, in general, non-odd and
  generates a $T$-periodic solution to the Lurye system.
\end{proof}
  % XXX It seems like the construction to generate an odd function
  % gives points like (yi,-ui) and (-yi,+ui) about (0,0) where both yi
  % and ui are non-zero.  Doesn't this mean that we need to directly
  % connect the points (yi,-ui), (-yi,+ui), and (0,0) by a nonzero
  % slope? In other words, the function will no longer be a
  % stair-step.
  % XXX - Just noticed that we need to be careful about the
  % use of exp(+jw) versus exp(-jw).  I think we should be using
  % the + version.  (Probably we could use either version as long
  % as we are consistent throughout).
%\vspace{0.05in}
If we restrict our attention to odd nonlinearities, i.e. $\phi\in S_{0,\infty}^\text{odd}$, the phase condition must be modified as follows:
\vspace{0.05in}

\begin{theorem}\label{th:1a}
  Let $G\in\RH$ and integers $0<\alpha<\beta$ be given.  Assume
  $\alpha$ and $\beta$ are co-prime. Define the frequency
  $\omega:=\frac{\alpha \pi}{\beta}$ with corresponding period 
  $T=2\beta$ if $\alpha$ is odd and $T=\beta$ if $\alpha$ is even. There exists $\phi \in S^\text{odd}_{0,\infty}$ such
  that the Lurye system has a non-trivial $T$-periodic solution if
  \begin{equation}\label{eq:5a}
    \pi-\frac{\pi}{2\beta} \le
    \angle G(e^{j\omega})\le \pi+\frac{\pi}{2\beta}.
  \end{equation} 
\end{theorem}
\vspace{0.05in}
\begin{proof}
  The statement with $\alpha$ odd follows from the proof of
  Theorem~\ref{th:1}.  If $\alpha$ is even then use the method in the
  proof of Theorem~\ref{th:1} to construct sequences
  $\{u_i\}_{i=0}^{\beta-1}$ and $\{y_i\}_{i=0}^{\beta-1}$.  Next,
  append the data to include both $(y_i,-u_i)$ and $(-y_i,u_i)$ for
  $i=0,\ldots,\beta-1$.  The phase condition in \eqref{eq:5a} can be
  used to show that the appended data satisfies
  Equation~\ref{eq:monotonecond}. Hence the data can be interpolated
  by a monotone nonlinearity $\phi$ (Lemma~\ref{lem:2}). Moreover, the
  appended data is odd and hence $\phi \in S^\text{odd}_{0,\infty}$
  (Lemma~\ref{lem:odd}).  The appended data is only used in the
  function interpolation and the Lurye system will have a
  $\beta$-periodic solution with only
  $\{(y_i,-u_i)\}_{i=0}^{\beta-1}$.
\end{proof}

`% XXX - ORIGINAL PROOF
% The statement with $\alpha$ odd follows from the proof of
% Theorem~\ref{th:1}.  If $\alpha$ is even then use the methodology in
% the proof of Theorem~\ref{th:1} to construct sequences
% $\{u_i\}_{i=0}^\beta$ and $\{y_i\}_{i=0}^\beta$. Next, construct an
% odd nonlinearity to interpolate the data $(y_i,-u_i)$ and
% $(-y_i,u_i)$ for all $i=0,1,...,\beta-1$.  The new set of points can
% be seen as the projection of $2\beta$ equidistantly spaced points
% around the unit circle. As a result, the phase condition to ensure
% that Equation~\ref{eq:monotonecond} holds is the same as in the case
% $\alpha$ odd.  By Lemma~\ref{lem:odd}, the resulting nonlinearity
% belongs to $S^\text{odd}_{0,\infty}$.

  % If $\alpha$ is even then Theorem~\ref{th:1} requires the phase of
  % $G(e^{j\omega})$ to lie within $\pi\pm \frac{\pi}{\beta}$.
  % Theorem~\ref{th:1a} requires the tigher phase constraint of
  % $\pi\pm \frac{\pi}{2\beta}$. This tighter phase constraint enables
  % the interpolation condition \eqref{eq:monotonecond} to be satisfied.

\subsection{Construction for $S_{0,k}$ with $k<\infty$}
\label{sec:mainS0k}

Consider a Lurye system of $(G,\phi)$ where $\phi$ is slope restricted
with $k<\infty$. The loop transformation in Figure~\ref{fig:1} maps to
a Lurye system $(\tilde{G},\tilde{\phi})$ where $\tilde{\phi}$ is
monotone.
\begin{lemma}
  \label{lem:looptr}
  The Lurye system with $G\in\RH$ and $\phi\in S_{0,k}$
  ($\phi\in S_{0,k}^\text{odd}$) has a periodic solution if and only
  if the Lurye system with $\tilde{G}:=G+1/k$ and
  $\tilde{\phi}\in S_{0,\infty}$
  ($\tilde{\phi}\in S_{0,\infty}^\text{odd}$) has a periodic solution.
\end{lemma}
\begin{proof}
  The proof follows from standard loop transformation arguments,
  see Chapter III, Section 6, in~\cite{D&V:75}.
\end{proof}
\begin{figure}[ht]
  \centering
  \ifx\JPicScale\undefined\def\JPicScale{0.25}\fi
\unitlength \JPicScale mm
\begin{picture}(180,175)(0,-110)
\thicklines
\put(-10,0){\vector(1,0){30}}
\put(23,0){\circle{6}}
\put(-5,5){0}

\put(26,0){\vector(1,0){50}}
\put(34,5){$u$}
\put(76,-15){\framebox(30,30){$G$}}

\put(76,30){\framebox(30,30){$1/k$}}
\put(51,0){\line(0,1){45}}
\put(51,45){\vector(1,0){25}}
\put(106,45){\line(1,0){33}}
\put(139,45){\vector(0,-1){42}}

\put(106,0){\vector(1,0){30}}
\put(139,0){\circle{6}}
\put(121,5){$y$}
\put(142,0){\line(1,0){20}}
\put(152,5){$\tilde{y}$}
\put(162,0){\vector(0,-1){45}}
\put(162,-48){\circle{6}}
\put(43,-20){\dashbox(105,85)}
\put(155,50){$\tilde{G}$}

\put(196,-48){\vector(-1,0){30}}
\put(186,-43){0}
\put(159,-48){\vector(-1,0){17}}
\put(139,-48){\circle{6}}
\put(136,-48){\vector(-1,0){30}}
\put(76,-63){\framebox(30,30){$\phi$}}
\put(121,-43){$y$}

\put(76,-108){\framebox(30,30){$1/k$}}
\put(51,-48){\line(0,-1){45}}
\put(51,-93){\vector(1,0){25}}
\put(106,-93){\line(1,0){33}}
\put(139,-93){\vector(0,1){42}}
\put(43,-113){\dashbox(105,85)}
\put(155,-105){$\tilde{\phi}$}

\put(76,-48){\line(-1,0){53}}
\put(23,-48){\vector(0,1){45}}
\put(26,-15){\line(1,0){7}}

% %\linethickness{0.3mm}
% \put(10,2.5){\line(1,0){10}}
% \put(10,2.5){\line(0,1){7.5}}
% \put(10,10){\vector(0,1){0.12}}
% \put(10,12.5){\circle{5}}

% %\linethickness{0.5mm}
% %\put(12,15){\line(1,0){2}}
% \linethickness{0.3mm}
% \put(0,12.5){\line(1,0){7.5}}
% \put(7.5,12.5){\vector(1,0){0.12}}
% \put(12.5,12.5){\line(1,0){7.5}}
% \put(20,12.5){\vector(1,0){0.12}}
% \put(30,12.5){\line(1,0){10}}
% \put(40,5){\line(0,1){7.5}}
% \put(40,5){\vector(0,-1){0.12}}

% \put(40,2.5){\circle{5}}
% \put(42.5,2.5){\line(1,0){7.5}}
% \put(42.5,2.5){\vector(-1,0){0.12}}
% \put(30,2.5){\line(1,0){7.5}}
% \put(30,2.5){\vector(-1,0){0.12}}

% \put(20,10){\line(1,0){10}}
% \put(20,15){\line(1,0){10}}
% \put(20,10){\line(0,1){5}}
% \put(30,10){\line(0,1){5}}

% \put(20,0){\line(1,0){10}}
% \put(20,5){\line(1,0){10}}
% \put(20,0){\line(0,1){5}}
% \put(30,0){\line(0,1){5}}

% \put(25,12.5){\makebox(0,0)[cc]{$G$}}
% \put(25,2.5){\makebox(0,0)[cc]{$\phi$}}

% \put(3,14){\makebox(0,0)[cc]{$0$}}
% \put(15,14){\makebox(0,0)[cc]{$u$}}
% \put(35,14){\makebox(0,0)[cc]{$y$}}
% \put(12.5,9){\makebox(0,0)[cc]{$-$}}
% \put(47,4){\makebox(0,0)[cc]{$0$}}
% %\put(35,4){\makebox(0,0)[cc]{$e_2$}}
% %\put(15,4){\makebox(0,0)[cc]{$h$}}
% %\put(35,14){\makebox(0,0)[cc]{$y$}}
\end{picture}
  \caption{Loop transformation for a Lurye systems}
  \label{fig:1}
  \vskip-5mm
\end{figure}
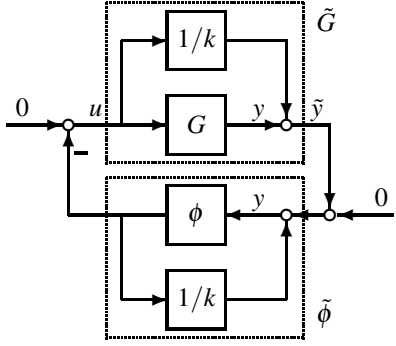

\vspace{0.05in}
\begin{proposition}\label{prop:1}
  Let $G\in\RH$ and integers $0<\alpha<\beta$ be given.  Assume
  $\alpha$ and $\beta$ are co-prime. Define the frequency
  $\omega:=\frac{\alpha \pi}{\beta}$ with corresponding period
  $T=2\beta$ if $\alpha$ is odd and $T=\beta$ if $\alpha$ is even.
  There is $\phi \in S_{0,k}$ with $k<\infty$ such that the Lurye
  system has a non-trivial $T$-periodic solution if
  \begin{align}
    \label{eq:prop1cond}
    R(\omega)+\frac{1}{k}<0 \mbox{ and }
    \frac{ -|I(\omega)|}{R(\omega)+1/k} \le 
    \tan\left(\frac{\pi}{T}\right), 
  \end{align}
  where $R(\omega)=\hbox{Re}\left\{G(e^{j\omega})\right\}$ and
  $I(\omega)=\hbox{Im}\left\{G(e^{j\omega})\right\}$.  
\end{proposition}
\begin{proof}
  If \eqref{eq:prop1cond} holds then $\tilde{G}:=G+1/k$ satisfies the
  phase conditon in \eqref{eq:5}. By Theorem~\ref{th:1}, there
  is a $\tilde{\phi}\in S_{0,\infty}$ such that the Lurye system of
  $(\tilde{G},\tilde{\phi})$ has a non-trivial solution.  This
  implies, by Lemma~\ref{lem:looptr}, that there is a
  $\phi\in S_{0,k}$ such that the Lurye system of $(G,\phi)$ has a
  non-trivial solution.
\end{proof}
The destabilizing nonlinearity with the smallest slope bound $\bar{k}$
is obtained when the second constraint in \eqref{eq:prop1cond} holds
with equality. Solving this equality for $\bar{k}$ yields:
\begin{align}
  \bar{k} =\frac{-\tan\left(\frac{\pi}{T}\right)}{R(\omega)\tan\left(\frac{\pi}{T}\right)+|I(\omega)|} 
\end{align}
If $G$ itself satisfies the phase condition in \eqref{eq:5} then
$\bar{k} \ge 0$. If $\bar{k}<0$ then no destabilizing nonlinearity
exists.  Finally, let $\{\tilde{y}_i,-u_i\}_{i=0}^{T-1}$ be the data
interpolated by $\tilde{\phi}$. The nonlinearity $\phi$ is obtained,
after loop transforming back, by interpolating
$\{\tilde{y}_i-u_i/k,-u_i\}_{i=0}^{T-1}$. The nonlinearity $\phi$ is
no longer multi-valued after the loop transformation.

\begin{proposition}\label{prop:2}
  Let $G\in\RH$ and integers $0<\alpha<\beta$ be given.  Assume
  $\alpha$ and $\beta$ are co-prime. Define the frequency
  $\omega:=\frac{\alpha \pi}{\beta}$ with corresponding period
  $T=2\beta$ if $\alpha$ is odd and $T=\beta$ if $\alpha$ is even.
  There is $\phi \in S_{0,k}^\text{odd}$ with $k<\infty$ such that the Lurye
  system has a non-trivial $T$-periodic solution if
  \begin{align}
    \label{eq:prop2cond}
    R(\omega)+\frac{1}{k}<0 \mbox{ and }
    \frac{ -|I(\omega)|}{R(\omega)+1/k} \le 
    \tan\left(\frac{\pi}{2\beta}\right) 
  \end{align}
  where $R(\omega)=\hbox{Re}\left\{G(e^{j\omega})\right\}$ and
  $I(\omega)=\hbox{Im}\left\{G(e^{j\omega})\right\}$.  
\end{proposition}
\vspace{0.05in}

The proof is similar to that given for Proposition~\ref{prop:1} and is
omitted.  Moreover, we can solve for the smallest $\bar{k}^\text{odd}$
for which there is a destabilizing $\phi \in S^\text{odd}_{0,k}$.

\section{Discussion on the Kalman Conjecture}\label{Sec:Kal}

The constructed nonlinearity is valid for each ($\alpha$,$\beta$)
where the phase condition is satisfied at the frequency
$\omega=\frac{\alpha\pi}{\beta}$. This provides an upper bound
$\bar{k}$ on the stability boundary $k_{AS}$ for the absolute
stability problem.  The Nyquist gain provides an alternative upper
bound using the class of linear gains.
\begin{definition}[Nyquist gain]
  The Nyquist gain of $G\in\RH$, denoted $k_N$, is the supremum of the
  set of gains $k$ such that the feedback interconnection between $G$
  and $K$ is stable for all $K\in[0,k]$.
\end{definition}
\vspace{0.05in} 
The constructed nonlinearity only provides new information if
$\bar{k}<k_N$. To clarify further, recall the Discrete-Time Kalman 
Conjecture (DTKC) is that $k_N=k_{AS}$ as stated next.
\begin{conjecture}[DTKC~\cite{Kalman:1957,Carrasco:15}]
  The Lurye system with $G$ and any $\phi\in S_{0,k}$ is stable if and
  only $k<k_N$.
\end{conjecture}
\vspace{0.05in} 
Our nonlinear construction does not provide any valuable information beyond the Nyquist value for plants where $k_{ZF}\simeq k_N$. However, as DTKC is false in general~\cite{Heath:15}, the Nyquist gain is a conservative upper bound. Our construction becomes relevant for the plants used in absolute stability literature such as the examples in~\cite{Carrasco:20}, where there is a significant gap between $k_{ZF}$ and $k_N$ (see Tables~I and~II in~\cite{Zhang:20}). For all the six examples in~\cite{Zhang:20}, our construction leads to counterexamples of the DTKC, i.e. $\bar{k}<k_N$ (see Table~III in~\cite{Zhang:20}).

\section{Numerical examples}

\subsection{Example  with $\alpha$ odd and  $k=\infty$}

To illustrate the main results, first consider artificially
constructed plants. Let $\alpha=1$ and $\beta=5$ so that
$\omega=\pi/5$. The periodic solutions have period $T=10$. 
Consider a plant $G_1$ with $G_1(e^{j\omega})=-e^{j\frac{\pi}{25}}$.
This plant satisfies the phase condition in Equation~\ref{eq:5} of
Theorem~\ref{th:1}. The input and output of $G_1$ are
$U_T = Re\{V_T\}$ and $Y_T = Re\{ G_1(e^{j\omega}) V_T\}$ where:
\begin{align*}
V_T & :=\begin{bmatrix}
1 & e^{j\frac{\pi}{5}} & e^{j\frac{2\pi}{5}}& e^{j\frac{3\pi}{5}} & 
\cdots & e^{j\frac{7\pi}{5}} & e^{j\frac{8\pi}{5}}& e^{j\frac{9\pi}{5}} 
\end{bmatrix}^\top \\
& :=\begin{bmatrix}
1 & e^{j\frac{\pi}{5}} & e^{j\frac{2\pi}{5}}& e^{j\frac{3\pi}{5}} & 
\cdots & e^{-j\frac{3\pi}{5}} & e^{-j\frac{2\pi}{5}} & e^{-j\frac{\pi}{5}}
\end{bmatrix}^\top 
\end{align*}
Figure~\ref{fig:4} plots the vectors $V_T$ (red) and
$G_1(e^{j\omega})V_T$ (blue) in the complex plane.  The projection of
these points onto the real axis corresponds with the input-output data
$U_T$ and $Y_T$.  In this example $\alpha$ is odd. Note that the
points in $V_T$ (i) are equidistantly spaced around the unit circle,
(ii) are symmetric about both the real and imaginary axis, (iii) and
there is a rotational symmetry of $\pi$.  These are the key
properties claimed in Theorem~\ref{th:1}. The points in  
$G_1(e^{j\omega})V_T$ are shifted slightly counterclockwise.
Figure~\ref{fig:5} shows the interpolated function (blue)
obtained from $(Y_T,-U_T)$ using Equation~\ref{eq:phiinterp}.
This function is odd and passes through $\phi(0)=0$.

Next consider a plant $G_2$ with
$G_2(e^{j\omega})=-e^{j\frac{\pi}{10}}$ and the same
$(\alpha,\beta,\omega)$ as given above.  This plant satisfies the
phase condition in Equation~\ref{eq:5} but with equality, i.e. the
phase condition is tight.  The input and output of $G_2$ are
$U_T = Re\{V_T\}$ and $Y_T = Re\{ G_2(e^{j\omega}) V_T\}$ where $V_T$
is the same as above. Figure~\ref{fig:4} plots the vectors $V_T$ (red)
and $G_2(e^{j\omega})V_T$ (green) in the complex plane.  The
projection of these points onto the real axis corresponds with the
input-output data $U_T$ and $Y_T$. Note that the green data has points
of the form $(a \pm j b)$ for some $(a,b)$. Projecting these points to
the real axis results in repeats in the entries of $Y_T$. As a result,
the interpolation $\phi$ is multivalued with a stair-step shape as
shown in Figure~\ref{fig:5}.

\begin{figure}
  \centering
  \includegraphics[width=0.88\linewidth]{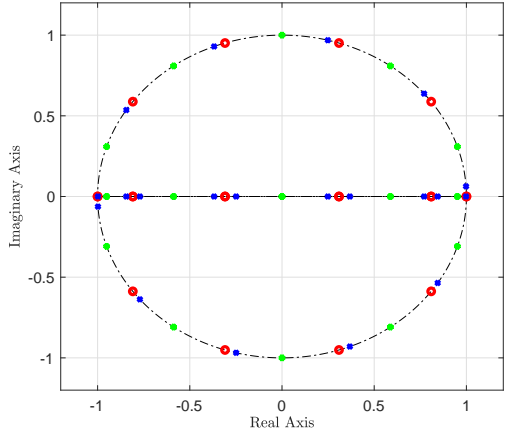}
  \caption{Points $V_T$, $G_1(e^{j\omega})V_T$, and $G_2(e^{j\omega})V_T$
    in the complex plane.}
  \label{fig:4}
  \vskip-5mm
\end{figure}

\begin{figure}
  \centering
  \includegraphics[width=0.88\linewidth]{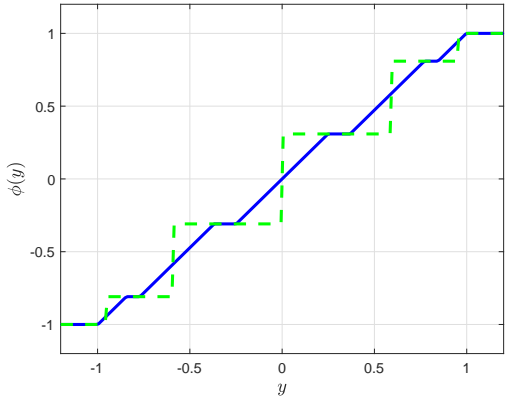}
  \caption{Interpolated nonlinearities for $G_1$ (blue) and
    $G_2$ (green).}
  \label{fig:5}
    \vskip-7mm
\end{figure}

\subsection{Example  with $\alpha$ even  and  $k=\infty$}

Let $\alpha=2$, $\beta=3$, hence $\omega=2\pi/3$. The periodic
solutions have period $T=3$.  Consider two different plants:
e.g. $G_1(e^{j\omega})=-e^{j\frac{\pi}{6}}$ and
$G_2(e^{j\omega})=-e^{j\frac{\pi}{3}}$. The input and outputs 
are $U_T = Re\{V_T\}$ and $Y_T = Re\{ G_i(e^{j\omega}) V_T\}$ 
$(i=1,2)$ where:
\begin{align*}
V_T& :=\begin{bmatrix}
1 & e^{j\frac{2\pi}{3}} & e^{-j\frac{2\pi}{3}}
\end{bmatrix}^\top \\
G_1(e^{j\omega}) V_T & =\begin{bmatrix}
e^{-j\frac{5\pi}{6}} & e^{-j\frac{\pi}{6}} & e^{j\frac{\pi}{2}} 
\end{bmatrix}^\top \\ 
G_2(e^{j\omega}) V_T & =\begin{bmatrix}
e^{-j\frac{2\pi}{3}} & 1 & e^{-j\frac{2\pi}{3}} 
\end{bmatrix}^\top.
\end{align*}

In this example we illustrate the interpolated nonlinearities. If we
consider the set $S_{0,\infty}$, we see that $G_1$ is not a limiting
case since it has finite slope, whereas $G_2$ is a limiting case as it
is multi-valued, see Figure~\ref{fig:6}. Moreover, the interpolated
nonlinearity is non-odd and it requires a shifting as explained in the
proof to obtain $\phi(0)=0$.  This shifting procedure is demonstrainted
in the following section.

On the other hand, if we reduce our attention to $S_{0,\infty}^\text{odd}$, by ensuring oddness, $G_1$ becomes a limiting case as it is multi-valued, see Figure~\ref{fig:7}. In addition, the required odd nonlinearity for $G_2$ is not monotone. However, it does not contradict Theorem~\ref{th:1a}, as condition~\ref{eq:5a} is not satisfied for $G_2$.

\begin{figure}
	\centering
	\includegraphics[width=0.88\linewidth]{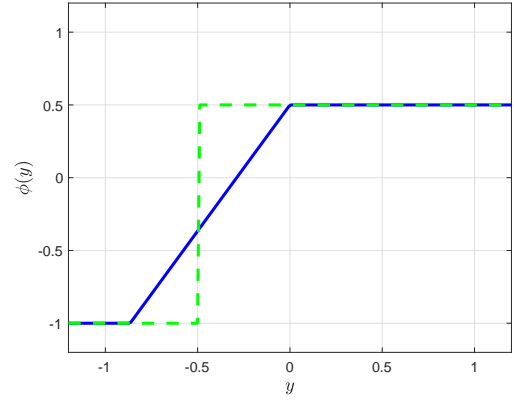}
	\caption{Interpolated nonlinearities required for the Lurye system to have a periodic behaviour for $G_1$ and $G_2$. As the pair of points are nonodd, the interpolated nonlinearities are nonodd and do not cross the origin.}
	\label{fig:6}
	\vskip-5mm
\end{figure}

\begin{figure}
	\centering
	\includegraphics[width=0.88\linewidth]{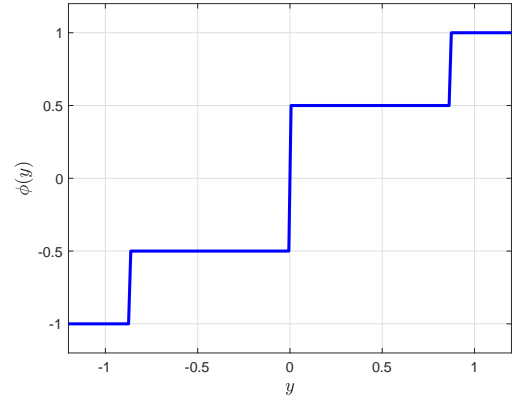}
	\caption{Interpolated nonlinearities required for the Lurye system to have a periodic behaviour for $G_1$. In this case, $G_1$ is multi-values and we cannot ensure the existence of a nonlinearity within $S_{0,\infty}^\text{odd}$ resulting in a periodic behaviour. For $G_2$, the interpolated nonlinearity does not belong to $S^\text{odd}_{0,\infty}$.}
	\label{fig:7}
	\vskip-5mm
\end{figure}

\subsection{Examples with $k<\infty$}

Consider the following system:
\begin{equation}\label{eq:10}
G(z)=\frac{z}{z^2-1.8z+0.81}.
\end{equation}
This plant has been used in~\cite{Carrasco:15,Heath:15} as a
second-order counterexample of the discrete-time Kalman Conjecture.
The feedback interconnection of $G$ and a (linear) gain $k$ is stable
if $0\leq k<3.61$. A 4-periodic cycle was constructed for a
slope-restricted nonlinearity with maximum slope $k=2.1$.

As mentioned in the introduction, Zames-Falb multipliers can be used to compute a lower bound
on $k_{AS}$.  Using the convex search in
\cite{Carrasco:20} yields multipliers that guarantee the stability for
all $\phi\in S_{0,\underline{k}_1}$ with $\underline{k}_1 = 1.3028317$
and for all $\phi\in S^{odd}_{0,\underline{k}_2}$ with
$\underline{k}_2 = 1.3511322$. We use the results in this paper
to construct destabilizing nonlinearities.  This construction
provides an upper bound $\bar{k} \ge k_{AS}$. For this plant
the upper bounds are close to the Zames-Falb lower bounds and hence
the conservatism in either bound is small.

First consider the class of non-odd nonlinearities.  Apply
Proposition~\ref{prop:1} using a large combination of values for
$\alpha$ and $\beta$.  We find that the minimum value of $\bar{k}$ is
obtained for $\alpha=2$ and $\beta=7$. For these values,
Proposition~\ref{prop:1} ensures periodic behaviour for all
$k\ge1.3028373$. The required nonlinearity is depicted in
Figure~\ref{fig:3}. To obtain this nonlinearity, we have to use
Equation~\eqref{eq:2}. For this particular plant, the DC gain of the
loop transformed plant is $\tilde{G}(1)=100+1/\bar{k}\simeq 100.7676$
and the shifting constant is $\xi=1.5985\times 10^{-3}$.

%As an illustration of the shifting procedure, Figure~\ref{fig:2} is included. 

\begin{figure}
  \centering
  \includegraphics[width=0.88\linewidth]{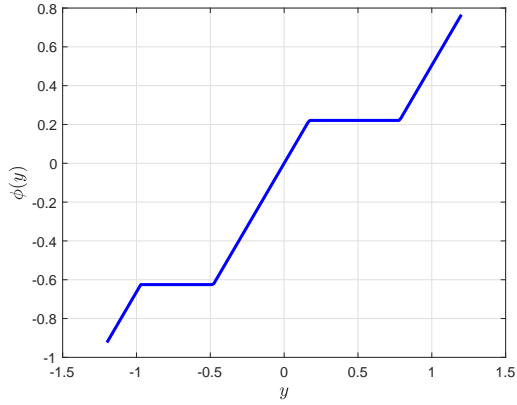}
  \caption{Interpolated nonlinearity $\phi\in S_{0,1.3028373}$.
    The Lurye system with $\phi$ and $G$ in \eqref{eq:10}
    exhibits a periodic behaviour.}
  \label{fig:3}
  \vskip-5mm
\end{figure}

%\begin{figure}
%	\centering
%	\includegraphics[width=1\linewidth]{nonlinearity6}
%	\caption{In blue, the interpolated staircase function obtained by using the real part of the eigenvector $V_2$ as input. In black, the line connecting the points $(0,0)$ and $(G(1),-1)$. The intersection provides the value $\xi=1.5985\times 10^{-3}$.}
%	\label{fig:2}
%	\vskip-3mm
%\end{figure}

Next consider the class of odd nonlinearities.  Apply
Proposition~\ref{prop:2} for a large combination of values for
$\alpha$ and $\beta$. We find that the minimum value of
$\bar{k}^\text{odd}$ is obtained for $\alpha=1$ and $\beta=3$. Then, $\omega=\frac{\pi}{3}$ and $T=6$. For these values,
Proposition~\ref{prop:1} ensures periodic behaviour for all
$k\geq1.3575410$. The required nonlinearity is depicted in
Figure~\ref{fig:8}. 

\begin{figure}
  \centering
  \includegraphics[width=0.88\linewidth]{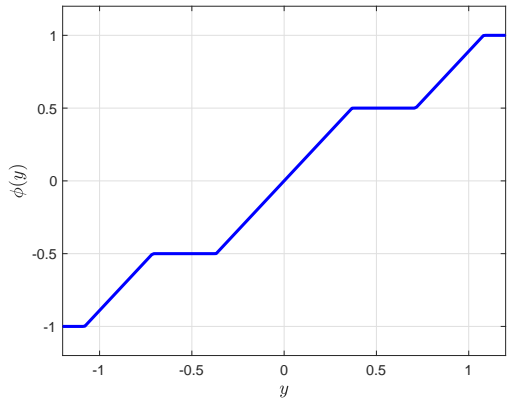}
  \caption{Interpolated nonlinearity $\phi\in S^\text{odd}_{0,1.3575410}$. The Lurye system with $\phi$ and $G$ in \eqref{eq:10} exhibits a periodic behaviour.}
  \label{fig:8}
  \vskip-5mm  
\end{figure}

\section{Conclusions}
This paper shows the connection between frequency-domain duality
conditions for Zames-Falb multipliers developed in~\cite{Zhang:20} and
periodic behaviour of the Lurye system for slope-restricted
nonlinearity. We develop an analytical construction for destabilizing
nonlinearities. For all examples in~\cite{Carrasco:20}, the construction yields systematic counterexamples
of the discrete-time Kalman conjecture, and therefore less conservative upper bounds for absolute stability.

\bibliographystyle{IEEEtran}
\bibliography{references} 
 
\end{document}